\documentclass[a4paper, 11pt]{article}
\usepackage{amssymb}
\usepackage{amsfonts}
\usepackage{amsmath}
\usepackage{hyperref}
\usepackage{xcolor}
\usepackage{yfonts}
\usepackage{graphicx,epstopdf}
\usepackage{tikz}

\usepackage{tikz-3dplot}

\usepackage{lscape}

\usepackage{xcolor}
\usepackage{yfonts}
\usepackage{hyperref}
\usepackage{float}
\usepackage{enumitem}

\usetikzlibrary{shapes.geometric, calc}%draw an n-gon
\usetikzlibrary{decorations.markings,shapes.arrows, shapes.symbols}%draw shape decoration
\usetikzlibrary{decorations.pathreplacing}

\newtheorem{theorem}{Theorem}[section]
\newtheorem{proposition}[theorem]{Proposition}
\newtheorem{lemma}[theorem]{Lemma}
\newtheorem{corollary}[theorem]{Corollary}

\newtheorem{definition}[theorem]{Definition}
\newtheorem{example}[theorem]{Example}

\newenvironment{proof}[1][Proof]{\textbf{#1.} }{\ \rule{0.5em}{0.5em}}

\newcommand{\ceil}[1]{\left \lceil{#1}\right \rceil}
\newcommand{\PG}{P}

\begin{document}

\author{Yuri Muranov \\
University  of  Warmia and \\
Mazury in Olsztyn, Olsztyn\\
Poland \\
\and 
Anna Muranova \\
University  of  Warmia and \\
Mazury in Olsztyn, Olsztyn\\
Poland \\
}
\title{Homology of  graph burnings}
\date{}
\maketitle

\begin{abstract} In this paper   we study graph burnings using  methods of algebraic topology.  We prove that the  time function  of a burning is a graph map to a path graph.  Afterwards,  we define  a category whose  objects are  graph  burnings and  morphisms  are graph maps  which commute with   the time functions of the burnings.  In this category we study relations between burnings of different graphs and, in particular, between 
burnings of  a graph and its  subgraphs.    For every graph, we  define a simplicial complex, arising  from the set of all the burnings, which we call a configuration space of the burnings.  Further,  simplicial structure of the configuration space gives  burning homology  of the graph. We  describe properties of the configuration space  and the burning homology theory. In particular,  we prove that the one-dimensional skeleton of the configuration space of a graph $G$  coincides with the complement graph of $G$. The results are illustrated with numerous examples. 
\end{abstract}

\bigskip


%\tableofcontents

\section{Introduction}\label{S1}
\setcounter{equation}{0}

The spread of  influence  in  network and similar network processes  can be studied using the  mathematical model of graph burning
 \cite{Bonato},  \cite{Burning_2014},  \cite{Bonato_0}.  The graph burning of a finite graph $G$  is defined as follows (see Definition~\ref{d3.1}).  Let $G=(V,E)$ be a graph and $t$ be a parameter (time) which is increasing
from $t=1$ and which  is  taking sequential integer values.  Every vertex of the graph $G$ can be burned or unburned. A state to be burned or unburned for a vertex is changed depending on the time $t\in \{1,2,\dots \}$ by the following rules: initially all 
vertices of $G$ are unburned. If a vertex $v\in V$ is burned in a time $t=t_0$ then it is burned in any time $t\geq t_0$. At any time $t\geq 1$ we burn one  vertex which is unburned in the time  $t$ if such a vertex exists. If a vertex $v\in V$ is burned in a time $t=t_0$  then in the time $t_0+1$ all neighbor vertices of $v$ becomes burned.  The process ends in a time $t=T$ when all vertices of $G$ are burned.  This time is called the \emph{end time}  of the burning process. We note that for the one-vertex graph there is only one burning process  with  $T=1$.   The vertices which we burn are called \emph{sources} and the vertex which we burn in a time $t$ is called the \emph{source at time $t$}.  Note that our definition slightly differs from the definition of burning that is given in the papers cited above. In particular, as follows from our definition, the distance  between two consequent source vertices  is greater or equal two.  The \emph{burning number} $b(G)$ of a graph  $G$ is the minimum of value $T$  for all burning processes for~$G$. Let $P_n$ be a \emph{path graph} with 
the set of vertices $\{1,\dots, n\}$ and with the set of edges
$\{i,i+1\}$ for $i=1, \dots, n-1$. For the burning process defined in  \cite{Bonato},  \cite{Burning_2014}, and  \cite{Bonato_0},  it is known that $b(P_n)=\ceil{\sqrt n}$ where for $r\in \mathbb R$ the number  $\ceil{r}\in \mathbb Z$ means the smallest integer greater than or equal to  $r$ \cite[Th. 2.9]{Bonato_0}. 

In this paper   we study graph burning using methods of algebraic topology in graph theory.  We consider only \emph{finite} graphs, i. e. those having a finite number of vertices.
We focus on studying a  set $\mathcal B(G)$ of all  burnings of a graph $G$ and on studying  relations between burnings of different graphs. In particular, we study relations between burnings of graph and its subgraphs. 
 We define a natural structure of a path graph $P_T$ on the discrete time
 $\{1, \dots, T\}$ in the process of burning which ends at the time~$T$. Afterwards we   prove that every burning process of $G$ defines an unique graph map $\lambda\colon G\to P_T$ and describe relation of this  map  to the  burnings of subgraphs of $G$.    

We define and study a configuration space $\Delta _G$ of  the set $\mathcal B(G)$ which  has a natural structure of a simplicial complex and describe its basic properties.  We prove that the one-dimensional skeleton of the configuration space  $\Delta_G$ coincides with the complement graph $\overline G$ of the graph $G$. Further, using this complex  we define  homology groups  for $\mathcal B(G)$  and describe properties of obtained burning homology theory.  We give numerous examples  to illustrate the obtained results.

\section{Preliminaries}\label{S2}
\setcounter{equation}{0}

In this section, we recall basic notions of  the graph theory, see e. g. 
\cite{Gary, MiHomotopy}.

\begin{definition}\label{d2.1}  \rm
 \emph{A graph} $G=(V_G,E_G)$ is a non-empty finite set $V_G$ of  \emph{vertices} together with a set $E_G$ of non-ordered pairs $\{v,w\}\in E_G$ of distinct vertices $v,w\in V_G$  called \emph{edges}. 
\end{definition}

 Sometimes we will omit 
the subscript $G$ from the set  $V_G$ of vertices and from the set $E_G$ of edges if the graph is clear from the context. 

Two different  vertices are \emph{adjacent} if they are connected by an edge. All vertices which are adjacent to vertex $v$ are called \emph{neighbors} of $v$.  A vertex is \emph{incident} with an edge 
 if the vertex is one of the endpoints of that edge. 
Two different edges $e, e^{\prime}\in E_G$  of a graph $G$ are called \emph{incident} if they have one common vertex.

\begin{definition}\label{d2.2}\rm  
We say that a  graph  $H$ is a  \emph{subgraph} of a graph $G$  and  we write $H\subset G$ if  $V_H\subset V_G$ and $E_H\subset E_G$. A subgraph 
$H\subset G$ is the \emph{induced subgraph} if for any edge $\{v,w\}\in E_G$,  such that $v, w\in V_H$ we have $\{v,w\}\in E_H$. In this case we write 
$H\sqsubset G$.  For a subgraph $H\subset G$ we denote by $\widehat{H}
$ the induced subgraph which has the same set of vertices as $H$. 
\end{definition}

\begin{definition} \label{d2.3}\rm A \emph{graph map $f\colon G\to H$}  from a graph
$G$\emph{\ }to a graph $H$ is a map of vertices $f|_{V_G}\colon
V_{G}\rightarrow V_{H}$ such that for any edge   $\{v,w\}\in E_G$,  we have   
$\{f(v),f(w)\}\in E_H$  or $f(v) =f(w)\in V_H$. The map
$f$ is called a \emph{homomorphism}  if $\{f(v),f(w)\}
\in E_H$ for any $\{v,w\}\in E_G$. 
\end{definition}

For a graph $G$,  we  denote by $\operatorname{Id}_G\colon G\to G$ the \emph{identity} map (homomorphism)  that is the identity map on the set of vertices and  on the set  of edges. 

Thus we obtain a category  $\mathbf G$ whose objects  are  finite graphs
and whose morphisms are graph maps  and   a category  $\mathbb G$  whose objects  are  finite graphs
and whose morphisms are homomorphisms.

\begin{definition}\label{d2.4}  \rm For $n\geq 1$,  a \emph{path} in a graph $G=(V,E)$ is an alternating sequence $v_0, a_1, v_1, a_2, \dots , v_{n-1}, a_n, v_n$ of vertices 
$v_i\in V$ and edges $a_i\in E$ such that $a_i=\{v_{i-1}, v_i\}$ for $i=1, \dots, n$.  
For $n=0$,  a path is given by a vertex $v_0\in V$. The integer $n$ is the \emph{length} of the path.  The vertex $v_0$ is  the \emph{origin}  and the vertex $v_n$ is the \emph{end} of the path. A path $v_0, a_1, v_1, a_2, \dots , v_{n-1}, a_n, v_n$ is   \emph{closed} if $v_0=v_n$. A closed path  is  a  \emph{circuit } if all edges are pairwise distinct.   A \emph{cycle} in a graph is a non-empty circuit in which only the first and last vertices are equal.
\end{definition} 

\begin{definition}\label{d2.5} \rm  (i) A  graph  $G=(V,E)$ is \emph{connected}
if for any two distinct vertices $v, w\in V$ there is a path  for which $v_0=v$ is the  origin and $v_n=w$ is the end of the path.

(ii) A  connected graph $G$ is a \emph{tree} if it is \emph{acyclic},  that is it does not  contain  nontrivial cycles.  

(iii) A graph $G$ is \emph{complete} if  every two distinct vertices of $G$  are adjacent.  We denote by $K_n$ a complete graph with $n$ vertices. In particular, we denote by $K_1$ the one vertex graph.

(iv) A \emph{complete bipartite  graph} is a graph whose vertices can be divided into two disjoint sets of vertices $V$ and $W$ such that the set of edges $E$ of the graph is given by the set of all pairs of vertices $\{v, w\}$ such that $v\in V, w\in W$. If the sets $V$ and $W$  consist of $n$ and $m$ elements, respectively,  the complete bipartite graph is denoted by $K_{n,m}$. 

(v) A graph is {\em edgeless}, if its set of edges is empty. We denote by $\overline K_n$ the edgeless graph with $n\ge 1$ vertices.

(vi)
 The {\em complement graph} of a graph $G=(V,E)$ is a graph denoted by $\overline G$ on the same set of vertices $V$ such that two distinct vertices of $\overline G$  are adjacent if and only if they are not adjacent in $G$. 
\end{definition}

Note that the edgeless graph $\overline K_n$ is the complement of the complete graph $K_n$ for every $n\ge 1$, which explains the notations for edgeless graphs.

\begin{definition}\label{d2.6}\rm For a graph $G=(V,E)$, the \emph{distance} $d_G(v,w)$ between two vertices $v,w\in V$ is the length of the shortest path connecting these vertices.  If there is no path connecting  vertices $v,w\in V$ then we set $d_G(v,w)=\infty$.  
\end{definition}

\begin{definition}\label{d2.7}  \rm  Let $G=(V_G,E_G)$ be a graph. 
For an integer $n\geq 0$, the \emph{n-th closed neighborhood} of a vertex $v\in V_G$ is an induced subgraph with the set of vertices $\{w\in V_G\;|\; \, d_G(v,w)\leq n\}$. We denote this neighborhood by~$N_n(v)$. 

\end{definition}

\begin{definition}\label{d2.8} \rm Let $H_i=(V_i, E_i)$,  $i=1,2$ be two subgraphs of a graph $G$.

(i)  The \emph{union} $H_1\cup H_2$ is the subgraph of $G$ with  the set of vertices $V_1\cup V_2$ and with the set of edges $E_1\cup E_2$. 

(ii) The \emph{intersection} $H_1\cap H_2$ is the subgraph of $G$ with  the set of vertices $V_1\cap V_2$ and with the set of edges $E_1\cap E_2$. 

(iii) The \emph{induced union} $H_1\Cup H_2$ is the induced subgraph of $G$ induced by the union $H_1\cup H_2$. It  equals  to the induced subgraph $\widehat{H_1\cup H_2}$. 

\end{definition} 

\section{Morphisms of  graphs and graph burning}\label{S3}
\setcounter{equation}{0}

In this section we define  a \emph{graph burning} of a graph $G$. Afterwards, we prove that there exists a graph map of the graph $G$ to a path graph $P_n$, defined by the burning. Further, we describe  various relations between burnings of a graph and burnings of its induced subgraphs. At the end of the section  we compute burning time for various burnings of several classes of graphs. 

 Let $G=(V_G,E_G)$ be a graph and 
$S_G=(v_1, \dots, v_k)$ be a non-empty ordered  set of vertices. For $1\leq j\leq k+1$,
consider a set of induced  subgraphs $\mathcal N_j\sqsubset G$ defined by 
\begin{equation}\label{3.1}
\mathcal N_j= \begin{cases} N_{j-1} (v_1)\Cup \dots \Cup N_1(v_{j-1})\Cup N_0(v_j) &  \text{for} \ \  1\leq j\leq k,\\
 N_{k} (v_1)\Cup N_{k-1}(v_2)\Cup \dots \Cup N_1(v_{k}) & \text{for} \  j=k+1.\\
\end{cases} 
\end{equation}
%We note that 
%\begin{equation*}
%\begin{matrix}
%\mathcal N_1=\{v_1\}, \ \mathcal N_2=N_1(v_1)\Cup \{v_2\},  \ \mathcal N_3=N_2(v_1)\Cup N_1(v_2)\Cup \{v_3\}, \ \dots \ . 
%\end{matrix}
%\end{equation*}

For $2\leq j\leq k+1$, let  $U_j$ be  the induced subgraph 
of $G$ given by   
\begin{equation}\label{3.2}
U_j=N_{j-1} (v_1)\Cup N_{j-2}(v_2)\Cup \dots \Cup N_1(v_{j-1}).
\end{equation}
We note that  for $2\leq j\leq k$, the set of vertices $V_{U_j}$ of the graph $U_j$ is obtained from the set of vertices $V_{\mathcal N_j}$ of the graph $\mathcal N_j$  by deleting the vertex $v_j$ and $U_{k+1}=\mathcal N_{k+1}$.

It follows immediately from (\ref{3.1}) and (\ref{3.2}) that we have the following  commutative diagram of inclusions of induced subgraphs 
\begin{equation}\label{3.3}
\begin{matrix}
\mathcal N_1 &\sqsubset & \mathcal N_2& \sqsubset &\mathcal N_3&\sqsubset \dots \sqsubset  &\mathcal N_k&\sqsubset&\mathcal N_{k+1}&\sqsubset  G\\
 && \sqcup&  &\sqcup &\dots & \sqcup&&||&\  \ \ || \\
                  &  &U_2&\sqsubset& U_3&\sqsubset\dots \sqsubset  &U_k&\sqsubset&U_{k+1}&\sqsubset  G\\
&& \sqcup&  &\sqcup &\dots & \sqcup&&\sqcup&\ \ \  \sqcup \\
                  &  &\mathcal N_1&\sqsubset& \mathcal N_2&\sqsubset\dots \sqsubset  &\mathcal N_{k-1}&\sqsubset&\mathcal N_{k}&\sqsubset  \mathcal N_{k+1}\\
\end{matrix}
\end{equation}
in which each row is a  filtration of the graph  $G$ by induced subgraphs.

\begin{definition}\label{d3.1} \rm Let $G=(V,E)$ be a graph and 
$S_G=(v_1, \dots, v_k)$ be a non-empty ordered  set of vertices. 
 The sequence $S_G$ is called a \emph{burning sequence}  if  in  diagram (\ref{3.3}) $v_j \notin U_j$ for $2\leq j\leq k$ and  $U_{k+1}=G$. 
 The elements of $S_G$ are called \emph{burning sources}. The vertex $v_i$ is called the \emph{source at time $i$}. 
\end{definition}

 A burning sequence $S_G$ defines a commutative diagram  (\ref{3.3}) and since in this case $\mathcal N_{k+1}=U_{k+1}=G$
we obtain a filtration of $G$ 
\begin{equation}\label{3.4}
\mathcal N_1 \sqsubset  \mathcal N_2\sqsubset \dots \sqsubset  \mathcal N_k\sqsubset \mathcal N_{k+1}=G. 
\end{equation}
Let $\mathbb N$ be the set of positive integers. Define a function 
\begin{equation}\label{3.5} 
\lambda_G\colon V_G\to \mathbb N\;\mbox{ by }\; \lambda_G(v)=\min\{i \,| \, v\in \mathcal N_i\}.
\end{equation}
It follows from  the  filtration (\ref{3.4})   that   the function $\lambda_G$ is well defined. 

\begin{lemma}\label{l3.2} Let $S_G=(v_1, \dots, v_n)$ be a burning sequence of a graph $G=(V,E)$.  Let $T=\max\{\lambda_G(v)\;|\;  \, v\in V_G\}$ where $\lambda_G$ is defined by \eqref{3.5}. Then the function $\lambda_G$ is a surjective map from  the set $V$ of vertices to the set  $\{1,2,\dots, T\}$,  
where  $T=k$ or $T=k+1$. 
\end{lemma} 
\begin{proof} It follows from \eqref{3.4}  and (\ref{3.5}) that for any vertex $v\in V$ we have $\lambda_G(v)=j\in \{1,\dots, k+1\}$.  By (\ref{3.1}),    we obtain that the burning source $v_j\in \mathcal N_j$ for $1\leq j \leq k$. By Definition \ref{3.1}, $v_j\notin U_j$ for $2\leq j\leq k$ and, hence, by (\ref{3.3}), $v_j\notin \mathcal N_{j-1}$ for $2\leq j\leq k$. Hence  $\lambda_G(v_j)=j$ for any  source 
$j\in S_G$.  In the case   $\mathcal N_k=\mathcal N_{k+1}$ we obtain that $T=k$. Otherwise $\mathcal N_k\ne \mathcal N_{k+1}$ and   $T=k+1$. 
\end{proof}

 We call a pair $\mathbf B(G)\colon =(\lambda_G, S_G)$  the \emph{burning (process)} for the graph $G$. We call the integer $T$   the \emph{end time of the burning}. Further,   for a vertex $v\in V_G$ the value $\lambda_G(v)$ is called  the \emph{burning time of the vertex} $v$.   We write $\mathbf B^T(G)$   if the process ends in the  time $T$. We denote  by $\lambda(\mathbf  B^T(G))$ the function $\lambda_G$ which corresponds to the  process $\mathbf B^T(G)$.

\begin{lemma}\label{l3.3}
Let $G=(V,E)$ be a graph with a burning  $\mathbf  B(G)=(\lambda, S_G)$.    Let  $v,w\in V$ be  two adjacent vertices. If   $v\in \mathcal N_j$  for some $j$ then $w \in \mathcal N_{j+1}$.
\end{lemma}

\begin{proof}
By the definition  of $\lambda$ in (\ref{3.5})  and using (\ref{3.1}) we obtain that 
$$
v\in \mathcal N_j=N_{j-1} (v_1)\Cup N_{i-2}(v_2)\Cup \dots \Cup N_1(v_{j-1})\Cup N_0(v_j),
$$
Hence  there exists at least one number $i \, (1\leq i\leq j)$ such that $v\in~ N_{j-i}(v_i)$, that is $d_G(v_i, v)\leq j-i$. We have $d_G(v,w)=1$  since $\{v, w\}\in E$.  Hence, by the triangle inequality for the graph distance $d_G$ we obtain that 
$d_G(v_i,w)\leq  d_G(v_i,v)+ d_G(v,w) = j-i+1$. Afterwards,  by Definition  \ref{d2.7} of closed neighborhoods, we obtain that $w\in N_{j-i+1}(v_i)\subset~\mathcal N_{j+1}$.
\end{proof}

\begin{lemma}\label{l3.4}
Let $G=(V,E)$ be a graph and $\mathbf  B^T(G)$ be a graph burning with $\lambda=\lambda_G$.  Then for any two adjacent vertices $v,w\in V$  we have $|\lambda(v)-\lambda(w)|\in\{0,1\}$.
\end{lemma}

\begin{proof}
If $\lambda(v)=\lambda(w)$, then there is nothing to prove. Without loss of generality,  assume that $\lambda(w)>\lambda(v)$. Then $v\in\mathcal N_{\lambda(v)}$ and by Lemma \ref{l3.3} $w\in \mathcal N_{\lambda(v)+1}$, and hence, by (\ref{3.3}),  $\lambda(w) \le \lambda(v)+1$ from where the statement follows.
\end{proof}

\begin{theorem}\label{t3.5} Let $G=(V,E)$ be a graph and $\mathbf  B^T(G)$ be a graph burning. Then the function $\lambda_G=\lambda(\mathbf  B^T(G))$  defines a unique graph   map $G\to \PG_T$  which coincides 
with  $\lambda(\mathbf  B^T(G))$ on the set of vertices $V$ of the graph $G$. We continue to denote this graph map  by $\lambda=\lambda(\mathbf  B^T(G))$. 
\end{theorem}

\begin{proof} Recall that $P_T$ is the path graph, i.e the graph with the set of vertices $\{1,\dots , T\}$ and the set of edges $\{i, i+1\}$ for $i=1, \dots, T-1$. On the set of vertices,  the function $\lambda \colon V\to V_{P_T}$  is well defined by (\ref{3.5}).  Note that for any two vertices $i, j\in V_{P_T}$ there is an edge $\{i,j\}\in E_{P_T}$  if and only if $|i-j|=1$. For any edge $\{v,w\}\in E$, the vertices $v$ and $w$ are neighbors
and,   by Lemma \ref{l3.4},  $\lambda(v)=\lambda(w)$ or $|\lambda(v)- \lambda(w)|=1$ and in this case 
$\{\lambda(v), \lambda(w)\}\in E_{P_T}$. Hence, by Definition \ref{d2.3},  $\lambda$ is well defined on the set of edges. 
\end{proof}

\begin{definition}\label{d3.6}\rm Let $G$ be a graph and    $\mathbf  B^T(G)$   be a burning of $G$. The graph map (homomorphism) $\lambda=\lambda(\mathbf  B^T(G))$ is called the \emph{burning map (homomorphism)}.   
\end{definition}

For a graph $G$ we would like to  describe invariants of graph maps $\lambda\colon G\to \PG_n  \, (n\geq 1)$ which detect burning maps. At first we consider several examples.

\begin{example}\label{e3.7} \rm Consider the  graphs $G=(V_G,E_G)$ and $H=(V_H, E_H)$ in Figure \ref{GH} which have the same set of vertices $V_G=V_H=\{1,2,3,4,5\}$.  
\begin{figure}[H]
\centering
\begin{tikzpicture}

\node (0) at  (5, 3) {$\bullet$};
\node (0a) at  (5, 2.7) {$2$};
\node (1) at (3,3) {$\bullet$};
\node (1a) at (3,2.7) {$1$};
\node (2) at (3,4.8) {$\bullet$};
\node (3) at (3,5.11) {$3$};
\node (2x) at (5,4.8) {$\bullet$};
\node (3x) at (5,5.11) {$4$};
\node (2y) at (7.1,4.8) {$\bullet$};
\node (3y) at (7.1,5.11) {$5$};

\draw (2x) edge[ color=black!120, thick, -] (2);
\draw (0) edge[ color=black!120, thick, -] (1);
\draw (1) edge[ color=black!120, thick, -] (2);
\draw (2x) edge[ color=black!120, thick, -] (2);
\draw (2x) edge[ color=black!120, thick, -] (0);
\draw (2x) edge[ color=black!120, thick, -] (2y);

\node (11) at (8.7,3.9) {$H\colon $};
\node (111) at (2.5,3.9) {$G\colon $};

\node (6) at  (9.3, 3) {$\bullet$};
\node (6b) at  (9.3, 2.7) {$1$};
\node (7) at  (11.2, 3) {$\bullet$};
\node (7b) at  (11.2, 2.7) {$2$};
\node (2z) at (11.2,4.8) {$\bullet$};
\node (3z) at (11.2,5.11) {$4$};

\node (8) at  (9.3, 4.8) {$\bullet$};
\node (8b) at  (9.3,  5.11) {$3$};

\node (2t) at (13.1,4.8) {$\bullet$};
\node (3t) at (13.1,5.11) {$5$};

\draw (2z) edge[ color=black!120, thick, -] (2t);
\draw (2z) edge[ color=black!120, thick, -] (7);
\draw (6) edge[ color=black!120, thick, -] (8);
\draw (2z) edge[ color=black!120, thick, -] (8);
\draw (6) edge[ color=black!120, thick, -] (7);
\draw (7) edge[ color=black!120, thick, -] (8);
\end{tikzpicture}
  \caption{The graphs $G$ (left) and  $H$ (right) in Example \ref{e3.7}.}
\label{GH}
\end{figure}
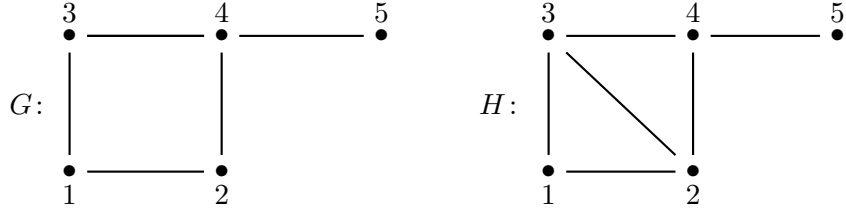
\noindent
Let $S_G=S_H=(1,5)$. Then we have the following  burning maps:
\begin{enumerate}[label=(\roman*)]
\item
$\lambda_G\colon V_G\to  V_{\PG_3}=\{1,2,3\}$ is given by $\lambda(1)=1, \lambda(2)=\lambda(3)=\lambda(5)=2, \lambda(4)=3$. The pair $(\lambda_G, S_G)$ defines a burning $\mathbf  B(G)$ with $T_G=3$  for which the burning map $\lambda_G$ is a homomorphism.
\item
$\lambda_H\colon V_H\to  V_{\PG_3}=\{1,2,3\}$  coincides with $\lambda_G$  on $V_H=V_G$ and  the pair $(\lambda_H, S_H)$ defines a  burning $\mathbf  B(H)$ with $T_H=3$ for which  the burning map $\lambda_H$ is  not a homomorphism.
\end{enumerate}
\end{example} 

\begin{lemma}\label{l3.8} Let a graph $G$ contains a \emph{triangle} subgraph which is isomorphic to the graph $T$ in Figure \ref{Triangle digraph}. Then the graph $G$ does not admit a burning homomorphism $\lambda_G$. 
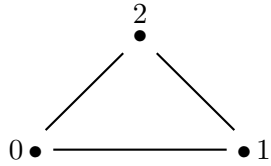
\begin{figure}[H]
\centering
\begin{tikzpicture}[scale=0.75]
\node (0) at  (-5.9, 3) {$0\, \bullet$};
\node (1) at (-1.9,3) {$\bullet\, 1$};
\node (2) at (-3.9,5) {$\bullet$};
\node (3) at (-3.9,5.4) {$2$};
\draw (0) edge[ color=black!120, thick, -] (2);
\draw (0) edge[ color=black!120, thick, -] (1);
\draw (1) edge[ color=black!120, thick, -] (2);
\end{tikzpicture}
  \caption{Triangle graph $T$.}
\label{Triangle digraph}
\end{figure}
\end{lemma}
\begin{proof} By Theorem \ref{t3.5} every burning defines a graph map $f\colon G\to P_n$ whose restriction gives a map $f|_{T}\colon T \to P_n$. The image of any such map  is  a one vertex or a one edge. In each case the image of at least one edge of $T$ will  be a vertex and the statement of the Proposition follows. 
\end{proof}

\begin{lemma}\label{l3.9}  Let a graph $G$  has a closed path of  length $n=2k+1\geq 3$. Then for every path graph $P$,  a homomorphism $\phi\colon G\to P$ does not exist.  
\end{lemma}

\begin{proof} We prove by an   induction in  $n=2k+1$ that a homomorphism $\phi\colon G\to P$ does not exist  for any graph $G$ which admits a closed path of length $n=2k+1$.

Firstly, it is obvious that the image of a closed path under the homomorphism is a closed path  of the same length. Therefore, it is enough to prove, that there is no a closed path of odd length in a path graph. We prove this by induction on $k$. For $k=1$ (i.e. $n=3$) the statement holds
by the proof of Lemma \ref{l3.8}.

Let us suppose that none of the path graphs admits a closed path of length $n=2(k-1)+1$. Assume, that there exists a path graph $P_m$, $m\in \Bbb N$, $m>0$, which admits a closed path $C=(v_0, a_1,\dots , v_{n-1}, a_n,  v_n=v_0)$ of length $n=2k+1\geq 5$. Let $M=\max \{v_i\;|\; v_i\in C\}$. Then there are two options:

(i) $M=v_0$,

(ii) $M= v_{i}$ and $i\ne 0$.

\noindent If $M=v_0$, then $v_1=v_{n-1}$ and $P_m$ admits a closed path $(v_1,a_2,\dots , v_{n-1})$ of length $n=2(k-1)+1$ which is a contradiction. If $M=v_i$, $i\ne {0}$, then $v_{i-1}=v_{i+1}$ and  $P_m$ admits a closed path $(v_0,a_1,\dots v_{i-1},a_{i+2},v_{i+2}\dots , v_{n})$ of length $n=2(k-1)+1$ which is again a contradiction. \end{proof}

\begin{theorem}\label{t3.10} Let a graph $G$ contains a closed path  of length $n=2k+1$, $n\geq~3$. Then the graph $G$ does not admit a burning homomorphism.
\end{theorem} 
\begin{proof}  By Theorem \ref{t3.5}, every burning $\mathbf B(G)=(\lambda, S_G)$ defines a 
 graph map $\lambda\colon G\to P_T$ where $P_T$ is the path graph.  Now the statement of the Theorem follows from Lemma \ref{l3.9}. 
\end{proof}

\begin{proposition}\label{p3.11} Fix a natural number $T\ge 1$ and consider burnings 
$\mathbf B^T(P_n)=(\lambda, S_{P_n})$ of  path graphs $P_n$ with the end time $T$.
\begin{enumerate}[label=$(\roman*)$]

\item  Then the maximum value of $n$  equals  $T^2$, i. e.  the longest path graph which can be burned  for the end time $T$ has $T^2$ vertices.  
\item 
If in additional $\lambda$ is a homomorphism, then  the maximum value of $n$  equals  $T^2-T+1$.
\end{enumerate}
\end{proposition}
\begin{proof}
(i) Firstly, we prove that  $n\le T^2$. By  Lemma \ref{l3.2}, we conclude 
that there are two possibilities for a burning  of a graph $P_n$ with  the end time $T$:

(1) a  burning of $P_n$ is given by a source sequence $S_{P_n}=(v_1,\dots, v_T)$ and 
\begin{equation}\label{3.6}
 \mathcal N_T= N_{T-1}(v_1)\Cup \dots \Cup N_{1}(v_{T-1})\Cup \mathcal N_0(v_T)=P_n,
\end{equation}
or, otherwise,  

(2)  a burning of $P_n$ is given by a source sequence $S_{P_n}=(v_1,\dots v_{T-1})$ and 
\begin{equation}\label{3.7}
 \ \mathcal N_{T}=N_{T-1}(v_1)\Cup  \dots \Cup N_{1}(v_{T-1})=P_n. \ \ \ \ \ \ \ \ 
\end{equation}
In the case (1), to obtain a maximum value $n$ we can take all the neighborhoods  in (\ref{3.6}) in such a way that pairwise distinct  neighborhoods in (\ref{3.6}) have empty intersection.  For a vertex  $v\in P_n$,  the maximal neighborhood 
$N_j(v)$  contains $2j+1$ vertices.  Hence, in the case (1),
\begin{equation}\label{3.8}
n\le \sum_{j=0}^{T-1}(2j+1)=T^2.
\end{equation}
In the case (2), using the same line of arguments   we obtain  that $n\le~T^2-1$.  Thus in both cases  $n\le T^2$. Finally, the burning sequence with $v_1=T$ and 
$$
v_j=\sum_{i=T-(j-1)}^{T-1}(2 i+1)+T-j+1 \ \ \text{for} \ \ 2\le j \le T,
$$
 defines the burning of the path graph  $P_{T^2}$.

(ii)   Now let $\lambda$ be a homomorphism. We consider this case similarly to the case (i).  At first we prove that the  equality (\ref{3.6}) leads to a contradiction. Note, that $n>1$, otherwise the statement is trivial. Since path graph $P_n$ is connected there is a vertex $v\in V_{P_n}$ such that $\{v, v_T\} \in E_{P_n}$. Moreover, $\lambda(v)=T-1$ since $\lambda$ is a homomorphism and $T$ is the maximal value of the map $\lambda$ on the set of vertices.  By Definition \ref{d3.1} we obtain
\begin{equation*}
v\in \mathcal N_{T-1}=N_{T-2}(v_1)\Cup N_{T-3}(v_2)\Cup \dots \Cup N_0(v_{T-1}). 
\end{equation*}
Thus, $v$ is a vertex of at least one of the graphs 
$N_{T-2}(v_1), \dots , N_0(v_{T-1})$.
  Since $d(v,v_T)=1$ we obtain that the vertex $v_T$ 
is the vertex of at least one of  the following  graphs 
$
N_{T-1}(v_1), \dots , N_1(v_{T-1})
$
and, hence, 
\begin{equation*}
v_T\in N_{T-1}(v_1)\Cup  N_{T-2}(v_2) \Cup \dots \Cup N_1(v_{T-1})= U_{T}. 
\end{equation*}
But, by Definition \ref{d3.1} and our assumption, $v_T\notin U_T$. Hence, we have obtained a contradiction in the case of the relation  (\ref{3.6}).

Thus,  we can consider only the relation  (\ref{3.7}). 
The vertices  of the set $S_{P_n}=\{v_1, \dots , v_{T-1}\}$ are pairwise distinct  elements of the set of vertices 
$V_{P_{n}}=\{1, \dots , n\}$ and hence they  have a natural  order "$<$". Thus we obtain a sequence of  vertices $w_i \in V_{P_n} \ (1\leq i\leq T-1)$ and a permutation  $\pi$  of elements $(1, \dots, T-1)$ such that 
\begin{equation}\label{3.9}
1\leq w_1< w_2< \dots <w_{T-1}\leq n     \ \ \ \text{and} \ \  w_i=v_{\pi(j)}. 
\end{equation} 
For any two consequent vertices $w_i< w_{i+1}$ in (\ref{3.9}),  we 
have two neighborhoods from (\ref{3.7})
$M_{i}=N_{T-\pi^{-1}(i)}(w_i)$  and $M_{i+1}=N_{T-\pi^{-1}(i+1)}(w_{i+1})$.   To obtain   a maximum value of $n$ from above,  these neighborhoods must have minimal intersections.  The case of $M_{i}\cap M_{i+1}=\emptyset$ is impossible since in this case the burning map $\lambda$  is not a homomorphism.   Let  $M_i\cap M_{i+1}$ consists from  exactly one vertex for $1\leq i\leq T-2$. Then  the neighborhoods in  (\ref{3.7}) have $T-2$ one-vertex intersections since  the number of elements of this cover is $T-1$.  Hence, by (2) of proof  (i), we obtain 
\begin{equation}\label{3.10}
n\le T^2-1-(T-2)=T^2-T+1. 
\end{equation}
In this case, the burning sequence with $v_1=T$ and 
$$
v_j=\sum_{i=T-(j-1)}^{T-1}(2 i)+T-j  \ \ \text{for} \ \  j\ge 2
$$
 defines the burning of the path graph  $P_{T^2-T+1}$.
\end{proof}

\begin{proposition}\label{p3.12}
Let $\mathbf B^T(P_n)=(\lambda, S_{P_n})$ be a burning for some path graph $P_n$ and $S_{P_n}=\{v_1, \dots, v_k\}$. Then the following holds:

 {\rm  (i)}  $n\le k^2+2k$ and  the equality is reached;

{\rm (ii)} if  in addition $\lambda$ is a homomorphism, then $n\le k^2+k+1$ and   the equality is reached.
\end{proposition}

\begin{proof}
{\rm  (i)} The proof follows from the proof of (i) in Proposition \ref{p3.11}. Indeed, in this case the burning time is either $T=k$ and \eqref{3.6} holds or the burning time is ${T=k+1}$ and \eqref{3.7} holds. Hence, $n=k^2$ or 
$n=T^2-1=k^2+2k$,  from where the statement follows.

{\rm  (ii)} can be treated similarly, using the proof of (ii) in Proposition \ref{p3.11}.
\end{proof}

\begin{corollary}\label{c3.13}  The burning number  $b(P_n)$ for the burning defined in Definition \ref{d3.1} of a path graph $P_n$ equals $\ceil{\sqrt n}$ where  $\ceil{r}\in \mathbb Z$ is  the smallest integer greater than or equal to  $r$. Moreover, this number coincides with the burning number  for the burning process defined in \cite{Burning_2014}, \cite[Th. 2.9]{Bonato_0} and \cite{Bonato}).
\end{corollary}

\begin{proof} The fact  that $P_n$ can be burned in a time $\ceil{\sqrt n}$ and can not be burned faster follows from the proof of (i) in Proposition \ref{p3.11}. 
\end{proof}

There is  a natural question about minimal length $n$ of a path graph $P_n$ which  can be burned   by a source sequence of length $k\geq1$.

\begin{proposition}\label{p3.14}  {\rm (i) } The minimum value $n$, for which there exists a burning $(\lambda, S_{P_n})$ of a path graph 
   $P_n$  using $k$ sources, equals  $2k-1$.   

{\rm (ii)} If, additionally, $\lambda$  is a burning homomorphism, then the  minimum  value $n$, for which there exists a  burning $(\lambda, S_{P_n})$    using $k$ sources, equals  $3k-2$.  
\end{proposition}

\begin{proof} (i)  Consider a burning $S_{P_n}=(v_1, \dots, v_k)\, (k\geq 1) $ and two consequent sources  $v_{j-1}, v_{j}$ of the burning. By Definition \ref{d3.1},  $v_j\notin U_j$ and, hence,  $v_j\notin~N_1(v_{j-1})$ by (\ref{3.2}). That is, the distance 
$d_{P_n}(v_j,v_{j-1})$ is greater or equal to two.   The minimal distance $d_{P_n}(v_j,v_{j-1})=2$ for every pair of consequent sources  is realized for the path graph $P_{2k-1}$ with the source sequence $v_i=2i-1$, $i=1,2,\dots, k$. 

(ii)  The consideration in similar to the case (i). In the case of a homomorphism $\lambda$ the distance 
$d_{P_n}(v_j,v_{j-1})$ is greater or equal to three. The minimal distance $d_{P_n}(v_j,v_{j-1})=3$ for every pair of consequent sources  is realized for the path graph $P_{3k-2}$ with the source sequence $v_i=3i-2$.
\end{proof}

\bigskip

Now we describe functorial relations between burnings of various graphs. Let $(\lambda_G, S_G)$ and  $(\lambda_H, S_H)$  be two  burnings  of  graphs $G$  and $H$, respectively,  with the source sequences  $S_G=(v_1, \dots, v_k)$  and $S_H=~(w_1, \dots, w_m)$  where $m\geq k\geq 1$. 
 For $2\leq j\leq  k+1$, let $U_j$   be the induced subgraph of $G$ defined in (\ref{3.2}).  For $2\leq j\leq  k+1$,  let 
$W_j\sqsubset H $  be the induced subgraph of $H$ defined  by
\begin{equation}\label{3.11}
W_j=N_{j-1} (w_1)\Cup N_{j-2}(w_2)\Cup \dots \Cup N_1(w_{j-1}).
\end{equation}

\begin{definition}\label{d3.15} \rm  Let  $G$, $H$ be graphs and
$\mathbf B(G)$,  $\mathbf B(H)$ be two  burnings  as above.
  A graph map $f\colon G\to H$ is called a \emph{morphism of the burnings}, and we write $f\colon \mathbf B(G)\to \mathbf B(H)$,   if  $k\leq m$,    $f(v_i)=w_i$ for  $1\leq i\leq k$,  and  there is a  commutative diagram 
of graph maps
\begin{equation}\label{3.12}
\begin{matrix}
G &\overset{f}\longrightarrow & H\\
\lambda_G\downarrow\ \ \  \ &&\ \ \ \downarrow\lambda_H\\
P_{T^G} &\overset {\tau}\longrightarrow & P_{T^H}, \\
\end{matrix}
\end{equation}
where $T^G$ and $T^H$ are the burning times for the burnings  of  the graphs $G$ and $H$,  respectively. Hence, in particular,  $f(U_j)\subset W_j$ for $2\leq j\leq k+1$.
\end{definition}

Now we  describe the properties of burning morphisms and afterwards consider   several examples.

\begin{proposition}\label{p3.16} Let $f\colon (\lambda_G, S_G)\to (\lambda_H, S_H)$  be a morphism of graph burnings with $S_G=(v_1,\dots, v_k)$ and $S_H=(w_1, \dots , w_m)$.  If the end time of the burning of the graph $G$ is $k$, then the graph map $\tau$ in diagram \eqref{3.12} is an inclusion. 
\end{proposition}
\begin{proof} Let $S_G=(v_1,\dots, v_k)$ and $S_H=(w_1, \dots , w_m)$ where $k\leq m$ by Definition \ref{d3.15}. It follows from the commutativity of diagram \eqref{3.12} that on the set of vertices  $\{1,\dots, k\}\subset V_{P_{T^G}}$ the function $\tau$ is an inclusion given by 
\begin{equation}\label{3.13}
\tau(i) =\tau(\lambda_G(v_i))=\lambda_H(f(v_i))= \lambda_H(w_i)=i.
\end{equation}
Hence, $\tau$ is an inclusion for the case $T^G=k$.   
\end{proof} 

\begin{lemma}\label{l3.17} {\rm (i) } Let $(\lambda_G, S_G)$ be a burning of a graph $G$ and $\operatorname{Id}_G \colon G\to G$ be the identity  map. Then $\operatorname{Id}_G$ defines  the identity morphism of the burnings   
$(\lambda_G, S_G) \to  (\lambda_G, S_G) $.

{\rm (ii) } Let $f\colon (\lambda_G, S_G)\to (\lambda_H, S_H)$ and $g\colon  (\lambda_H, S_H)\to  (\lambda_K, S_K)$ be two  morphisms of graph burnings.  Then the composition $g\circ f$ is a  morphism $(\lambda_G, S_G)\to  (\lambda_K, S_K)$
of graph burnings. 
\end{lemma}
\begin{proof} 
 {\rm (i) } It follows immediately from  Definition \ref{d3.15}.  
 
{\rm (ii) } Let $S_G=(v_1,\dots, v_k)$, $S_H=(w_1, \dots , w_m)$, and 
$S_K=(z_1, \dots , z_l)$. Then   $1\leq k\leq m\leq l$ by Definition \ref{d3.15}. Moreover, 
$$
(g\circ f)(v_i)= g(f(v_i))=g(w_i)=z_i \ \ \text{for} \ \ 1\leq i\leq k.
$$
Consider the diagram 
\begin{equation}\label{3.14}
\begin{matrix}
G &\overset{f}\longrightarrow & H&\overset{g}\longrightarrow & K\\
\lambda_G\downarrow\ \ \  \ &&\ \ \ \downarrow\lambda_H && \ \ \ \downarrow\lambda_K\\
P_{T^G} &\longrightarrow & P_{T^H} &\longrightarrow & P_{T^K}. \\
\end{matrix}
\end{equation}
in which the left square is commutative since $f$ is a burning morphism and the right square is commutative since $g$ is a burning morphism. Hence diagram 
(\ref{3.14}) is commutative and the statement is proved. 
\end{proof}

\begin{theorem}\label{t3.18}  There is a category $\mathfrak{B}$ whose objects are graph burnings  and whose morphisms are morphisms of graph burnings defined in Definition~\ref{d3.15}.
\end{theorem} 
\begin{proof} The category of graphs and graph maps is well defined. Now the proof follows from Lemma \ref{l3.17}. 
\end{proof}

\begin{example}\label{e3.19} \rm Consider a map $f\colon G\to H$ of graphs in Figure \ref{map1} which is given by $f(v_0)=w_0, f(v_1)=w_1, f(v_2)=f(v_3)=f(v_4)=w_2$.
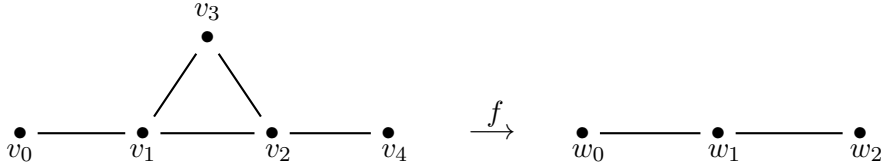
\begin{figure}[H]
\centering
\begin{tikzpicture}[scale=0.85]
\node (4) at  (1.1, 3) {$\bullet$};
\node (aa) at  (1.1, 2.7) {$v_0$};
\node (5) at  (6.8, 3) {$\bullet$};
\node (ba) at  (6.9, 2.7) {$ v_4$};

\node (0) at  (5, 3) {$\bullet$};
\node (0a) at  (5.1, 2.7) {$v_2$};
\node (1) at (3,3) {$\bullet$};
\node (1a) at (3,2.7) {$ v_1$};
\node (2) at (4,4.5) {$\bullet$};
\node (3) at (4,4.9) {$v_3$};
\draw (0) edge[ color=black!120, thick, -] (2);
\draw (0) edge[ color=black!120, thick, -] (1);
\draw (1) edge[ color=black!120, thick, -] (2);

\draw (1) edge[ color=black!120, thick, -] (4);
\draw (5) edge[ color=black!120, thick, -] (0);

\node (00) at (8.4,3) {$\longrightarrow$};
\node (11) at (8.45,3.3) {$f$};

\node (6) at  (9.8, 3) {$\bullet$};
\node (6b) at  (9.9, 2.7) {$ w_0$};
\node (7) at  (11.9, 3) {$\bullet$};
\node (7b) at  (12, 2.7) {$w_1$};
\node (8) at  (14.1, 3) {$\bullet$};
\node (8b) at  (14.2,  2.7) {$w_2$};
\draw (6) edge[ color=black!120, thick, -] (7);
\draw (7) edge[ color=black!120, thick, -] (8);
\end{tikzpicture}
  \caption{The map $f$ in Example \ref{e3.19}.}
\label{map1}
\end{figure}
Let  $\mathbf B(G)=(\lambda_G, S_G)$ where $S_G=(v_0, v_2)$ and $\mathbf B(H)=(\lambda_H, S_H)$ where $S_H=(w_0, w_2)$. Then $T_G=3$ and $T_H=2$. The map $f$ is a morphism of burnings $f\colon \mathbf B(G)\to \mathbf B(H)$ and the  graph map   $\tau\colon P_3\to P_2$ is given on the set of vertices by $\tau(1)=1, \tau(2)=\tau(3)=2$. 
\end{example} 

\begin{definition}\label{d3.20} \rm (i) Let  $G$ be a connected graph   and $H$ be a   connected subgraph  given with an inclusion $ i\colon H\to G$. Let $\mathbf B=(\lambda_G, S_G)$ be a burning of $G$.   The \emph{subgraph $H$ is $\mathbf B$-burned}  if there is a burning $(\lambda_H, S_H)$ of $H$ such that  $\lambda_H=\lambda_G|_{H}$ and the inclusion $i$ is a morphism of burnings 
$ (\lambda_H, S_H) \to (\lambda_G, S_G)$. 

(ii) The $\mathbf B$-burned  subgraph $H\subset G$ is \emph{$\mathbf B$-burned minimal}  if there is no a $\mathbf B$-burned subgraph $H_1\subset G$ such that  $H_1\varsubsetneq H$.  
\end{definition} 

\begin{example}\label{e3.21} \rm  Consider the graph $G$ in Figure \ref{BG}.  
\begin{figure}[H]
\centering
\begin{tikzpicture}[scale=0.85]
\node (4) at  (1.1, 3) {$\bullet$};
\node (aa) at  (1.1, 2.7) {$v_0$};
\node (5) at  (6.8, 3) {$\bullet$};
\node (ba) at  (6.9, 2.7) {$ v_5$};
\node (6) at  (8.9, 3) {$\bullet$};
\node (6ba) at  (9, 2.7) {$ v_6$};

\node (0) at  (5, 3) {$\bullet$};
\node (0a) at  (5.1, 2.7) {$v_2$};
\node (1) at (3,3) {$\bullet$};
\node (1a) at (3,2.7) {$ v_1$};
\node (2) at (5,4) {$\bullet$};
\node (3) at (5.1,4.4) {$v_3$};
\node (2b) at (5,1.8) {$\bullet$};
\node (3b) at (5.1,1.5) {$v_4$};

\draw (5) edge[ color=black!120, thick, -] (6);
\draw (5) edge[ color=black!120, thick, -] (2);
\draw (0) edge[ color=black!120, thick, -] (1);
\draw (1) edge[ color=black!120, thick, -] (2);

\draw (1) edge[ color=black!120, thick, -] (2b);
\draw (1) edge[ color=black!120, thick, -] (4);
\draw (5) edge[ color=black!120, thick, -] (0);
\draw (5) edge[ color=black!120, thick, -] (2b);

\end{tikzpicture}
  \caption{Graph $G$ in  Example \ref{e3.21}.}

\label{BG}
\end{figure}
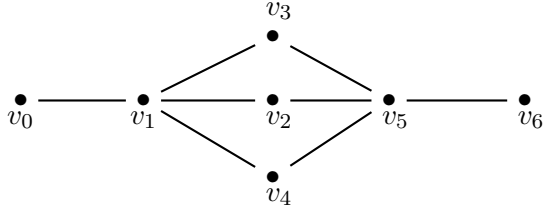
The  burning $\mathbf B(G)=(\lambda_G, S_G)$ is defined by $S_G=(v_0, v_5)$. The  function $\lambda=\lambda_G$ is given on the set of vertices by 
$
\lambda(v_0)=1, \ \lambda(v_1)=\lambda(v_5)=2$ and $\lambda(v_2)=\lambda(v_3)=\lambda(v_4)=\lambda(v_6)=3.
$
All the minimal $\mathbf B$-burned subgraphs $H_1$, $H_2$ and $H_3$ of $G$ are presented in Figure~\ref{MB1}.

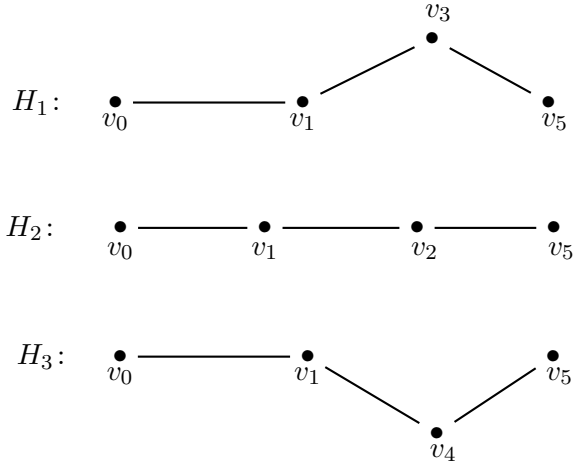
\begin{figure}[H]
\centering
\begin{tikzpicture}[scale=0.85]
\node (H1) at  (8.1, 3) {$H_1\colon \ \ $};
\node (4x) at  (9.1, 3) {$\bullet$};
\node (aax) at  (9.1, 2.7) {$v_0$};
\node (1x) at (12,3) {$\bullet$};
\node (1ax) at (12,2.7) {$ v_1$};
\node (2x) at (14,4) {$\bullet$};
\node (3x) at (14.1,4.4) {$v_3$};
\node (5x) at  (15.8, 3) {$\bullet$};
\node (bax) at  (15.9, 2.7) {$ v_5$};
\draw (1x) edge[ color=black!120, thick, -] (4x);
\draw (1x) edge[ color=black!120, thick, -] (2x);
\draw (5x) edge[ color=black!120, thick, -] (2x);
\end{tikzpicture}
\bigskip
\bigskip

\begin{tikzpicture}
\node (H2) at  (0.1, 3) {$H_2\colon \ \ $};
\node (4) at  (1.1, 3) {$\bullet$};
\node (aa) at  (1.1, 2.7) {$v_0$};
\node (1) at (3,3) {$\bullet$};
\node (1a) at (3,2.7) {$ v_1$};
\node (0) at  (5, 3) {$\bullet$};
\node (0a) at  (5.1, 2.7) {$v_2$};
\node (5) at  (6.8, 3) {$\bullet$};
\node (ba) at  (6.9, 2.7) {$ v_5$};
\draw (1) edge[ color=black!120, thick, -] (4);
\draw (0) edge[ color=black!120, thick, -] (1);
\draw (5) edge[ color=black!120, thick, -] (0);
\end{tikzpicture}
\bigskip
\bigskip

\begin{tikzpicture}[scale=0.85]
\node (H3) at  (8.1, 3) {$H_3\colon \ \ $};
\node (4x) at  (9.1, 3) {$\bullet$};
\node (aax) at  (9.1, 2.7) {$v_0$};
\node (1x) at (12,3) {$\bullet$};
\node (1ax) at (12,2.7) {$ v_1$};
\node (2x) at (14,1.8) {$\bullet$};
\node (3x) at (14.1,1.5) {$v_4$};
\node (5x) at  (15.8, 3) {$\bullet$};
\node (bax) at  (15.9, 2.7) {$ v_5$};
\draw (1x) edge[ color=black!120, thick, -] (4x);
\draw (1x) edge[ color=black!120, thick, -] (2x);
\draw (5x) edge[ color=black!120, thick, -] (2x);
\end{tikzpicture}
 \caption{The minimal $\mathbf B$-burned subgraphs $H_1, H_2, H_3$ in Example \ref{e3.21}.}
\label{MB1}
\end{figure}

\end{example} 

\begin{theorem}\label{t3.22} Let $\mathbf B=(\lambda, S_G)$ be a  burning of a graph $G$. Then every minimal $\mathbf B$-burned subgraph is a tree. 
\end{theorem}

\begin{proof}  Let $P_T$ be the image 
of $\lambda$ and $H\subset G$ be a minimal  $\mathbf B$-burned subgraph.    
Let us suppose, that $H$ is not a tree. Since $H$ is connected it   contains a
cycle $c=(v_0, a_1, \dots , v_{n-1}, a_n, v_0)$. Let $C\subset H$ be  the subgraph that is given by the cycle $c$. The restriction 
$\lambda|_{C}\colon C\to P_T$  is a graph map to the path graph and, hence, the image contains a maximal vertex $t=\lambda(v_k)\in V_{P_T}$ for a vertex  $v_k\in V_C$.  Since $C$ is a cycle, there exist  two distinct edges 
$\{x,v_k\},\{v_k,y\}\in E_C$ which are incident to $v_k$.  Since $t$ is maximum $\lambda_C$ on the set of vertices and by Lemma \ref{l3.4}, we obtain that 
$t-1\leq \lambda_C(x), \lambda_C(y)\leq t=\lambda_C(v_k)$.  Thus  burnings of the vertices $x$ and $y$ are not induced by the burning of the vertex $v$. Now we  delete  one of the edges  $\{x,v_k\},\{v_k,y\}$  from the graph $H$. Thus   we obtain  a  connected  subgraph $H_1\subsetneqq H$ with the same burning sequence $S_G$.  We obtain a contradiction with the condition that  $H$ is minimal.  The Theorem is proved. 
\end{proof}
\begin{example}\label{e3.23} \rm  Consider the  graph  $G$ in Figure \ref{Tr} with the burning $\mathbf B(G)$ that is given by $S_G=(v_0, v_4, v_6)$. Then there are  three  minimal $\mathbf B$-burning subgraphs $H_1$, $H_2$, and $H_3$ which  are induced subgraphs given by following sets  of vertices:
$V_{H_1}=\{v_0, v_1, v_2, v_3, v_4,v_6\}$,  $V_{H_2}=\{v_0, v_1, v_2, v_4, v_5,v_6\}$,  and $V_{H_3}=\{v_0, v_1, v_3, v_4, v_5,v_6\}$, respectively. The graphs $H_1$ and $H_2$ are not isomorphic and they are presented in Figure \ref{H1}. 
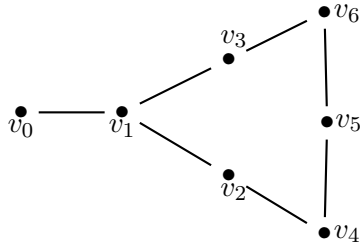
\begin{figure}[H]
\centering
\begin{tikzpicture}[scale=0.7]
\node (4) at  (1.1, 3) {$\bullet$};
\node (aa) at  (1.1, 2.7) {$v_0$};
\node (5) at  (6.8, 4.9) {$\bullet$};
\node (ba) at  (7.15, 4.88) {$\;v_6$};
\node (6) at  (6.8, 0.7) {$\bullet$};
\node (6ba) at  (7.15, 0.7) {$\;v_4$};
\node (1) at (3,3) {$\bullet$};
\node (1a) at (3,2.7) {$ v_1$};
\node (2) at (5,4) {$\bullet$};
\node (3) at (5.1,4.4) {$v_3$};
\node (2b) at (5,1.8) {$\bullet$};
\node (3b) at (5.1,1.5) {$v_2$};
\node (c) at  (6.85, 2.8) {$\bullet$};
\node (c1) at  (7.18, 2.8) {$\;v_5$};
\draw (5) edge[ color=black!120, thick, -] (2);
\draw (1) edge[ color=black!120, thick, -] (2);
\draw (1) edge[ color=black!120, thick, -] (2b);
\draw (1) edge[ color=black!120, thick, -] (4);
\draw (6) edge[ color=black!120, thick, -] (2b);
\draw (5) edge[ color=black!120, thick, -] (c);
\draw (6) edge[ color=black!120, thick, -] (c);
\end{tikzpicture}
  \caption{The graph $G$ in Example \ref{e3.23}. }
\label{Tr}
\end{figure}
\begin{figure}[H]
\centering
\begin{tikzpicture}[scale=0.8]
\node (4) at  (1.1, 3) {$\bullet$};
\node (aa) at  (1.1, 2.7) {$v_0$};
\node (1) at (2.5,3) {$\bullet$};
\node (1a) at (2.5,2.7) {$ v_1$};
\node (2b) at (3.9,2.3) {$\bullet$};
\node (3b) at (3.9,2) {$v_2$};
\node (2) at (3.9,3.5) {$\bullet$};
\node (3) at (3.9,3.9) {$v_3$};
\node (6) at  (5.3, 1.6) {$\bullet$};
\node (6ba) at  (5.67, 1.6) {$\; v_4$};
\node (5) at  (5.3, 3.95) {$\bullet$};
\node (ba) at  (5.67, 3.95) {$\; v_6$};
\draw (5) edge[ color=black!120, thick, -] (2);
\draw (1) edge[ color=black!120, thick, -] (2);
\draw (1) edge[ color=black!120, thick, -] (2b);
\draw (1) edge[ color=black!120, thick, -] (4);
\draw (6) edge[ color=black!120, thick, -] (2b);
\node (4d) at  (8.1, 3) {$\bullet$};
\node (aad) at  (8.1, 2.7) {$v_0$};
\node (1d) at (9.5,3) {$\bullet$};
\node (1ad) at (9.5,2.7) {$ v_1$};
\node (2bd) at (10.9,2.3) {$\bullet$};
\node (3bd) at (10.9,2) {$v_2$};
\node (6d) at  (12.3, 1.6) {$\bullet$};
\node (6bad) at  (12.67, 1.6) {$\; v_4$};
\node (5d) at  (12.3, 3.95) {$\bullet$};
\node (bad) at  (12.67, 3.95) {$\; v_6$};
\node (cd) at  (12.3, 2.7) {$\bullet$};
\node (c1d) at  (12.67, 2.7) {$\;v_5$};
\draw (1d) edge[ color=black!120, thick, -] (2bd);
\draw (1d) edge[ color=black!120, thick, -] (4d);
\draw (6d) edge[ color=black!120, thick, -] (2bd);
\draw (5d) edge[ color=black!120, thick, -] (cd);
\draw (6d) edge[ color=black!120, thick, -] (cd);
\end{tikzpicture}
  \caption{The subgraphs  $H_1$ (left) and $H_2$ (right) in Example \ref{e3.23}. }
\label{H1}
\end{figure}
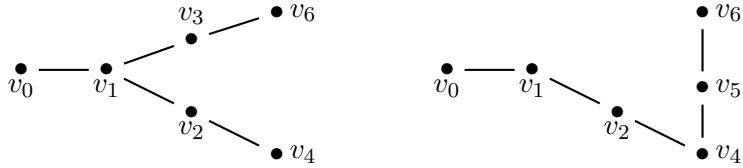
\end{example}

\section{Configuration space of  graph burnings}\label{S4}
\setcounter{equation}{0}

In this section,  we consider a set $\mathcal B(G)$ of all burnings of any graph $G$. We define  a configuration space   $\Delta_G$  which is a simplicial complex associated with the set  $\mathcal B(G)$ and prove that a one-dimensional skeleton of $\Delta_G$ coincides with the complement graph $\overline G$.  Afterwards,  we    describe  basic properties of $\Delta_G$ and describe relations between configuration spaces and standard  simplicial constructions.  
In this Section we will use standard definitions of simplicial and graph constructions \cite{Hatcher}, \cite{MiHomotopy},    \cite{Munkres}, \cite{Prasolov}. At first,   we recall  several basic  definitions  \cite{semi}, \cite{May}.

\begin{definition} \label{d4.1} \rm
A \emph{simplicial complex} $\Delta =(V,F)$ consists of a set $V$ of \emph{vertices} and a set $F$
of finite non-empty subsets of $V$ which are called \emph{simplexes}.  A simplex  $\sigma \in F$ containing $q+1 \, (q\geq 0) $ elements    is called a
$q$-\emph{simplex}  and the number $q$ is called the \emph{dimension} of $\sigma$ and we write  ($q=\operatorname{dim} \sigma$).  The set $F$ of simplexes must satisfy the following two properties:

(i) for every vertex $v\in V$  the set $\{v\}\in F$ is a simplex, 

(ii) if $\sigma \in F$ is a simplex then any   non-empty subset $\tau\subset \sigma$ is a simplex which is called  a \emph{face} of $\sigma$.

A simplex $\sigma\in \Delta$ is \emph{maximal} if it is not a face of any another simplex. 

A simplicial complex is $n$-\emph{dimensional} if it contains at least one $n$-dimen\-sional simplex but
no $(n+1)$-dimensional simplexes. 
\end{definition}

\begin{definition}\label{d4.2}\rm (i) A simplicial complex 
$\Delta_1=(V_1, F_1)$ is a \emph{subcomplex} of a simplicial complex  $\Delta=(V,F)$  if $V_1\subset V$ and $F_1\subset F$.

(ii) For $n\geq 0$, the \emph{$n$-skeleton}  $sk_n(\Delta)$ of a simplicial complex 
$\Delta=(V,F)$ is the subcomplex with the set of vertices $V$ and with the set of simplexes  $\{\sigma \in F| \operatorname{dim}\sigma \leq n\}$.
\end{definition} 

We note that every graph $G=(V,E)$ gives a   one-dimensional simplicial complex with the set $V$ of zero-dimensional simplexes and the set $E$ of one-dimensional simplexes.  Moreover, every  one-dimensional skeleton of a simplicial complex defines a graph with the same set of  vertices. 

\begin{definition} \label{d4.3} \rm  Let $\Delta=(V,F)$ and  $\Delta_1=(V_1,F_1)$
be simplicial complexes. A \emph{simplicial map}  $f\colon \Delta\to \Delta_1$ is given by a function $f\colon V\to V_1$ such that $f(\sigma)\in F_1$ for every simplex
$\sigma\in F$. 
\end{definition}

A simplicial complex is called \emph{finite}, if it has  finite number of vertices.
In this paper we  consider only finite  simplicial complexes.

\bigskip

Now we come back to the burning theory. Let $G=(V,E)$ be a graph. Denote by   $\mathcal  B(G)$ the set of all burnings of $G$. Since we consider only finite graphs, the set  $\mathcal  B(G)$
consists  of  finite number of pairs $(\lambda_G, S_G)$  defined in Section \ref{S3}. Every source sequence $S_G$ defines a set  $\widehat S_G= \{v_1, \dots, v_k\}\subset ~V$.    It follows from Definition \ref{d4.1} that a simplicial complex 
$\Delta=(V,F)$  is  defined  by the set $V$ of vertices and  a subset $F^m\subset F$ of maximal simplexes of $\Delta$.  We  will say that the set of maximal simplexes $F^m$ 
\emph{generates} the set $F$ of simplexes. 

\begin{definition}\label{d4.4}\rm  For a given graph $G=(V,E)$, we  define a \emph{burning configuration space of $G$} as a  simplicial complex $\Delta_G=(V_{\Delta_G}, F_{\Delta_G})$ where 
$V_{\Delta_G}=V$ and $F_{\Delta_G}$ is generated by the set of all maximal simplexes 
$$
F_{\Delta_G}^m=\{\widehat S_G\, |\, (\lambda_G, S_G)\in \mathcal B(G)\}.
$$  
\end{definition}

\begin{definition}\label{d4.5}\rm Let $G$ be a graph and    $H\subset G$ be a   connected subgraph with the inclusion $i\colon H\to G$. Consider a burning $\mathbf B(H)=(\lambda_H, S_H)$ of the graph $H$. We say that the burning $\mathbf B(H)$ 
\emph{admits an extension } to the  graph $G$  if there is a burning $\mathbf B(G)=(\lambda_G, S_G)$ such that 
$i\colon  (\lambda_H, S_H)\to (\lambda_G, S_G)$ is the morphism 
of burnings. 
\end{definition} 

\begin{example}\label{e4.6} \rm  Let us consider the  graph $G$ and its subgraph $H$ in Figure \ref{Ext}.  Consider a burning   $\mathbf B(H)$  with $S_H=(v_1, v_3)$. Then $\mathbf B(H)$ does not admit an extension to $G$.

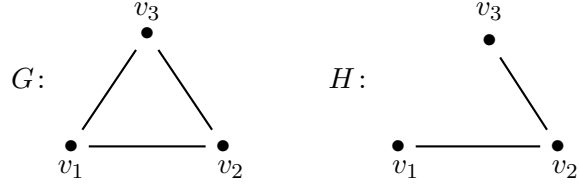
\begin{figure}[H]
\centering
\begin{tikzpicture}

\node (0) at  (5, 3) {$\bullet$};
\node (0a) at  (5.1, 2.7) {$v_2$};
\node (1) at (3,3) {$\bullet$};
\node (1a) at (3,2.7) {$ v_1$};
\node (2) at (4,4.5) {$\bullet$};
\node (3) at (4,4.79) {$v_3$};
\draw (0) edge[ color=black!120, thick, -] (2);
\draw (0) edge[ color=black!120, thick, -] (1);
\draw (1) edge[ color=black!120, thick, -] (2);

\node (11) at (6.7,3.9) {$H\colon $};
\node (111) at (2.5,3.9) {$G\colon $};

\node (6) at  (7.3, 3) {$\bullet$};
\node (6b) at  (7.4, 2.7) {$ v_1$};
\node (7) at  (9.4, 3) {$\bullet$};
\node (7b) at  (9.5, 2.7) {$v_2$};

\node (8) at  (8.5, 4.4) {$\bullet$};
\node (8b) at  (8.5,  4.79) {$v_3$};
\draw (6) edge[ color=black!120, thick, -] (7);
\draw (7) edge[ color=black!120, thick, -] (8);
\end{tikzpicture}
  \caption{The graphs $G$ (left) and  $H$ (right) in Example \ref{e4.6}.}
\label{Ext}
\end{figure}
   \end{example}

\begin{definition}\label{d4.7}\rm A graph $G$ is a \emph{burning extension of a subgraph $H$} if for every burning sequence $S_H$ there is a burning sequence $S_G$ such that $\widehat S_H\subset~\widehat S_G$. 
\end{definition}

\begin{lemma}\label{l4.8}  $\mathrm{(i)}$ Let $G=(V, E)$ be a graph, $v\in V$, and $K_1=\{v\}$ be a one-vertex subgraph with  the  natural inclusion $ i$. 
Then the unique burning $(\lambda_{K_1}, S_{K_1})$ admits an extension to $G$.
Moreover, the graph $G$ is a burning extension of the subgraph $K_1$. 

$\mathrm{(ii)}$ Let $G=(V, E)$ be a graph, $v,w\in V,$ and $\overline K_2=\{v, w\}$ be an edgeless subgraph with  the  natural inclusion $ i$. Let  
$\mathbf B=(\lambda_{\overline K_2}, S_{\overline K_2})$ be a 
 burning of $\overline K_2$. Then  $\mathbf B$ admits an extension to $G$ if an only if $v$ and $w$ are not adjacent, i.e. ${\overline K_2}$ is an induced subgraph. Moreover,    if $v$ and $w$ are not adjacent then    $G$ is a burning extension of $\overline K_2$. 
\end{lemma} 
\begin{proof} Follows from Definition \ref{d3.1}. 
\end{proof}

\begin{definition}\label{d4.9} \rm  (i) Let $G=(V_G, E_G)$ and $H=(V_H, E_H)$ be two graphs. The \emph{disjoint union} $G+H=(V,E)$  is a graph which  has as its vertices the set  $V=V_G\sqcup V_H$ and as its edges the set $E=E_G\sqcup E_H$.  

(ii) The \emph{iterated graph sum} $nG\, (n\geq 1)$ of a graph $G$  is the disjoint union of $n$ disjoint copies of $G$.
\end{definition} 

\begin{example}\label{e4.10} \rm 
Every  disjoint uinion $G+H$ is the burning extension of each of the graphs $G$ and $H$. 
\end{example}

\begin{proposition}\label{p4.11}  Let $G=(V, E)$ be a graph and $\Delta_G$ be its configuration space. Then the following holds:

 $\mathrm{(i)}$
 $0$-skeleton  $sk_0(\Delta_G)$  coincides with $V$, 
 
 $\mathrm{(ii)}$
 $1$-skeleton  $sk_1(\Delta_G)$  coincides with the complement graph $\overline G$.% of  the graph $G$.
\end{proposition} 
\begin{proof} $\mathrm{(i)}$ Follows immediately from Lemma \ref{l4.8}.

 $\mathrm{(ii)}$  Follows from Lemma \ref{l4.8}  since two adjacent vertices can not belong to the same burning sequence due to Definition \ref{d3.1}.
\end{proof}

\begin{corollary}\label{c4.12}  Let  $G$ be a graph and  $H=sk_1(\Delta_G)$  be a graph that is given by  the $1$-skeleton of its configuration space.  Then $sk_1(\Delta_H)=G$.  
\end{corollary} 
\begin{proof} The complement graph of $\overline G$ is $G$.
\end{proof} 

\bigskip

Recall that a \emph{clique}   in a graph $G$ is an induced subgraph $K$, which is complete. 
It follows from Proposition \ref{p4.13} below  that  in some sense the simplicial structure of $\Delta_G$  is opposite to the clique structure of  $G$.

\begin{proposition}\label{p4.13} Let $G$ be a graph and $K\subset G$ be a clique. Then any two distinct vertices $v, w\in V_K$ do not belong to a simplex $\sigma$ of $\Delta_G$. 
\end{proposition} 
\begin{proof} It follows from Definition \ref{d3.1} that two adjacent vertices can not belong to the same burning sequence. 
\end{proof}

\begin{proposition}\label{p4.15} Let a graph $G$ be  a burning extension of a graph $H$. Then the natural inclusion  $i\colon H\to G$ induces a simplicial map  $i_*\colon \Delta_H\longrightarrow \Delta_G$.
\end{proposition}

We note that the trees play a significant role in the burning theory, see e.g. \cite{Bonato_0}[Th.~2.4, Cor. 2.5] . 
We have good functorial relations between inclusions of trees  and burning extensions  defined in Definition \ref{d4.7} 

\begin{theorem}\label{t4.16}  Let $H$ and $G$ be trees and $H\subset G$. 
Then $G$ is the burning extension of the graph $H$. 
\end{theorem} 
\begin{proof} Consider a burning $\mathbf B(H)=(\lambda_H, S_H)$ where 
$S_H=(v_1, \dots, v_n), \  v_i\in V_H$ and $n\geq 2$. For $n=1$ the statement  follows from Definition \ref{d3.1}.   For  $2\leq j\leq n+1$ we denote by  
$U_j$  the  induced subgraph of $H$ defined  in (\ref{3.2}) where all  neighborhoods $N_k$ are  neighborhoods in $H$.   For  $2\leq j\leq n+1$, we define also an  induced subgraph of the graph $G$ by setting 
\begin{equation}\label{4.1}
U_j^G=N_{j-1}^G (v_1)\Cup N_{j-2}^G(v_2)\Cup \dots \Cup N_1^G(v_{j-1})
\end{equation}
where all  neighborhoods $N_k^G$ are the neighborhoods in $G$. 
We have evidently  $U_j\subset U_j^G$. By Definition \ref{d3.1}, we have  $v_j\notin U_j$ and $U_{n+1}=G$.  Now we prove that $v_j\notin U_j^G$. Indeed, 
let us suppose that $v_j\in U_j^G$. Then there exists at least one neighborhood 
$N_{j-k}^G(v_k)$ in (\ref{4.1})  such that $v_j\in N_{j-k}^G(v_k)$. Hence there  is a path of a minimal length $d\leq j-k$ from $v_k\in V_H$ to $v_j\in V_H$. Since $G$ is a tree  there is an unique path of minimal length $d$ from 
$v_k$ to $v_j$ which lays in  the subgraph $H$, since the tree is connected by the definition. Hence $v_j \in  N_{j-k}(v_k)$ and  
$v_j\in U_j$. We obtain a contradiction and, hence, $v_j\notin U_j^G$ for  $2\leq j\leq n+1$. Now if $U_{n+1}^G=G$ we have obtained that $(v_1, \dots, v_n)$ is the burning sequence of the graph $G$. Otherwise, we have  
$U_{n+1}^G\varsubsetneqq G$ and we can  add an  issue vertex $v_{n+1}\in V_G$ to the burning sequence $S_H$. Continuing this process we obtain a burning of $G$ which is a burning extension of $S_H$. 
\end{proof} 

\begin{corollary}\label{c4.17}   Every inclusion of trees $H\to  G$  induces   an inclusion of simplicial complexes $\Delta_H\to \Delta_G$. 
\end{corollary}

\begin{definition}\label{d4.18} \rm The \emph{cone over a simplicial complex} 
$\Delta=(V,F)$ is a simplicial complex 
$C(\Delta)=(V_{C(\Delta)}, F_{C(\Delta)})$ with the  vertex set 
 $V_{C(\Delta)}= V\sqcup\{v_*\}$  where $v_*$ is a new vertex, i.e. $v_*\not \in V$. 
 For $q>0$, a  $q$-simplex in $F_{C\Delta}$
is either a $q$-simplex of $\Delta$  itself, or it is  a disjont union $\sigma \sqcup \{v_*\}$ of the set  $ \{v_*\}$  with a $(q - 1)$-simplex  $\sigma$  of $\Delta$.
\end{definition}

\begin{definition}\label{d4.19} \rm The \emph{suspension over a simplicial complex} 
$\Delta=(V,F)$ is a simplicial complex 
$S(\Delta)=(V_{S(\Delta)}, F_{S(\Delta)})$ with the  vertex set 
 $V_{S(\Delta)}= V\sqcup\{v_*, w_*\}$  where $v_*, w_*$ are two  new distinct vertices, i.e. $v_*,w_*\not \in V$. For $q>0$, a 
$q$-simplex in $E_{S\Delta}$
is either a $q$-simplex of $\Delta$  itself, or it is  a disjoint union $\sigma \sqcup \{v_*\}$ of the set  $ \{v_*\}$  with a $(q - 1)$-simplex  $\sigma$  of $\Delta$, or
 it is  a disjoint union $\sigma \sqcup \{w_*\}$ of the set  $ \{w_*\}$  with a $(q - 1)$-simplex  $\sigma$  of $\Delta$.
\end{definition}

\begin{theorem}\label{t4.20} {\rm (i)} Let $G$ be a graph and $K_1$ be  the one-vertex graph. Then the  configuration space $\Delta_{G+K_1}$ coincides with the cone $C(\Delta_G)$ over the  configuration $\Delta_G$.

{\rm(ii)} Let $G$ be a graph  and $P_2$ be the path graph with two verices.  Then the configuration space $\Delta_{G+P_2}$ coincides with suspension 
$S(\Delta_G)$ over the configuration space  $\Delta_G$. 
\end{theorem}
\begin{proof} (i)  The graph   $G+K_1$  is the burning extension  of the graph $G$.  Now the result follows from Definitions \ref{d3.1}, \ref{d4.4}, and \ref{d4.18}. 

(ii)  Let $V_{P_2}=\{v_*,w_*\}$. The graph $G+P_2$ is the burning extension  of the graph $G$ and every burning sequence $S_{G+P_2}$  has the form 
$
(v_1, \dots, v_i, v_*, v_{i+1}, v_n)$ or 
$
(v_1, \dots, v_i, w_*, v_{i+1}, v_n)
$ 
where $(v_1, \dots, v_n)$ is a burning sequence  of $G$.  Now the result follows from Definitions \ref{d3.1}, \ref{d4.4}, and \ref{d4.19}.
\end{proof}

Below we will use the   \emph{geometric realization $|\Delta|$} of a simplicial complex $\Delta$ which is a \emph{polyhedron}  (see, for example, \cite[\S 1.5]{Matousek}).

\begin{theorem}\label{t4.21}  Let $\Pi_n=nP_2\, (n\geq 1)$ be an iterated sum of the path graph $P_2$. Then the geometric realization  $|\Delta_{\Pi_n}|$ of the burning configuration space  $\Delta_{\Pi_n}$ is homeomorphic to  $(n-1)$-dimensional sphere $S^{n-1}$.  
\end{theorem}

\begin{proof}  The configuration space of the graph $P_2$ consists of two distinct vertices.  Now induction in $n$ using Theorem \ref{t4.20}(ii) finishes the proof. 
\end{proof}

\begin{corollary}\label{c4.22} {\rm (i)} The configuration space of the graph $P_2+P_2$ is isomorphic to  the cyclic  graph $C_4$ with the set of vertices $\{0,1,2,3\}$ and the set of edges $\{\{0,1\}, \{1,2\}, \{2,3\}, \{3,0\}\}$.

{\rm (ii)} The geometric realization  of the configuration space of the graph $P_2+P_2+P_2$ is presented by the boundary of the  octahedron. 
\end{corollary}

\begin{example}\label{e4.23}\rm 
  Let $Q$ be a  graph cube given in Figure \ref{Q}.  Every burning of $Q$  has a sourse sequence consisting of two elements. Hence the configuration space $\Delta_Q$ is a one-dimensional complex. It is given  in Figure \ref{Q}.  By Corollary \ref{c4.12},  we obtain that $\Delta_{\Delta_Q}=Q$. 
 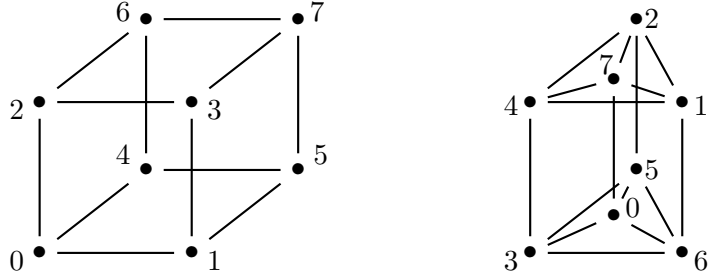
\begin{figure}[H]
\centering
\begin{tikzpicture}%[scale=0.85]
\node (1) at (4,3) {$\bullet$};
\node (10) at (3.7,2.9) {$0$};
\node (2) at (6,3) {$\bullet$};
\node (20) at (6.3,2.9) {$1$};

\draw (1) edge[ color=black!120, thick, ] (2);

\node (1a) at (4,5) {$\bullet$};
\node (10a) at (3.7,4.9) {$2$};
\node (2a) at (6,5) {$\bullet$};
\node (20a) at (6.3,4.9) {$3$};

\node (3a) at (7.4,4.1) {$\bullet$};
\node (3ab) at (7.7,4.3) {$5$};
\node (30a) at (5.4,4.1) {$\bullet$};
\node (30ab) at (5.1,4.3) {$4$};
\draw (3a) edge[ color=black!120, thick, ] (30a);

\node (3b) at (7.4,6.1) {$\bullet$};
\node (3bc) at (7.65,6.2) {$7$};

\node (30b) at (5.4,6.1) {$\bullet$};
\node (30bc) at (5.1,6.2) {$6$};
\draw (3b) edge[ color=black!120, thick, ] (30b);

\draw (3a) edge[ color=black!120, thick, ] (3b);
\draw (30a) edge[ color=black!120, thick, ] (30b);

\draw (1) edge[ color=black!120, thick, ] (30a);

\draw (1a) edge[ color=black!120, thick, ] (2a);

\draw (2a) edge[ color=black!120, thick, ] (2);

\draw (1a) edge[ color=black!120, thick, ] (1);

\draw (3a) edge[ color=black!120, thick, ] (2);

\draw (1a) edge[ color=black!120, thick, ] (30b);
\draw (2a) edge[ color=black!120, thick, ] (3b);
\end{tikzpicture} \ \ \ \ \ \ \ \ \ \ \ \ \  \ \
\begin{tikzpicture}[scale=1]
\node (1) at (4,3) {$\bullet$};
\node (10) at (3.75,2.9) {$3$};
\node (2) at (6,3) {$\bullet$};
\node (20) at (6.25,2.9) {$6$};

\node (30a) at (5.4,4.1) {$\bullet$};
\node (30ab) at (5.6,4.1) {$5$};

\node (1a) at (4,5) {$\bullet$};
\node (10a) at (3.75,4.95) {$4$};
\node (2a) at (6,5) {$\bullet$};
\node (20a) at (6.25,4.95) {$1$};
\node (3b) at (5.1,5.3) {$\bullet$};
\node (3bc) at (5,5.5) {$7$};

\node (30b) at (5.4,6.1) {$\bullet$};
\node (30bc) at (5.6,6.1) {$2$};

\node (3be) at (5.1,3.5) {$\bullet$};
\node (3bce) at (5.35,3.6) {$0$};
\draw (3be) edge[ color=black!120, thick, ] (30a);
\draw (3be) edge[ color=black!120, thick, ] (2);
\draw (3be) edge[ color=black!120, thick, ] (1);

\draw (2) edge[ color=black!120, thick, ] (30a);
\draw (2a) edge[ color=black!120, thick, ] (30b);
\draw (1a) edge[ color=black!120, thick, ] (3b);
\draw (3be) edge[ color=black!120, thick, ] (3b);
\draw (3b) edge[ color=black!120, thick, ] (30b);
\draw (30a) edge[ color=black!120, thick, ] (30b);
\draw (1) edge[ color=black!120, thick, ] (30a);
\draw (1a) edge[ color=black!120, thick, ] (2a);
\draw (2a) edge[ color=black!120, thick, ] (2);
\draw (1a) edge[ color=black!120, thick, ] (1);
\draw (1a) edge[ color=black!120, thick, ] (30b);
\draw (2a) edge[ color=black!120, thick, ] (3b);
\draw (1) edge[ color=black!120, thick, ] (2);

\end{tikzpicture}
\caption{A graph cube $Q$ (left) and its configuration space $\Delta_Q$ (right).}
\label{Q}
\end{figure}
\end{example}

\begin{example}\label{e4.25} \rm (i) For $n\geq 1$, let $K_n\, (n\geq 1)$ be a complete graph with $n$ vertices and $\overline K_n\, (n\geq 1)$  be an edgeless graph with $n$ vertices.  Then $\Delta_{K_n}=\overline K_n$ and $\Delta_{\overline K_n} $  is the simplicial complex  consisting of one $(n-1)$-dimensional simplex.  Moreover we have $sk_1(\Delta_{\overline K_n})=K_n$. The latest fact illustrates Proposition~\ref{p4.11}.

(ii) For   $n,m\geq 1$, let $K_n, K_m$  be two complete graphs and  $K_{n, m}$  be a  complete bipartite graph .  Every  source set of a burning of the graph  $K_n+K_m$ consists of a pair $(v,w)$  where 
$v\in V_{K_n}, w\in V_{K_m}$ or   $w\in V_{K_n}, v\in V_{K_m}$.  Moreover every such pair is the source set of 
$(K_n+K_m)$-burning. Hence $\Delta_{K_n+K_m}=K_{n,m}$. Similarly, we can see that  $\Delta_{K_{n,m}}=K_n +K_m$. 

(iii) Consider the path graph $P_5$ with five vertices $\{1,2,3,4,5\}$.
Then the geometric realization $|\Delta_{P_5}|$ of the configuration space $\Delta_{P_3}$ is presented in Figure \ref{P5} below. The shadowed triangle in Figure \ref{P5} corresponds to the  two-simplex in $\Delta_{P_5}$. Again, according to Proposition~\ref{p4.11}, $sk_1(\Delta_{P_5})=\overline P_5$.
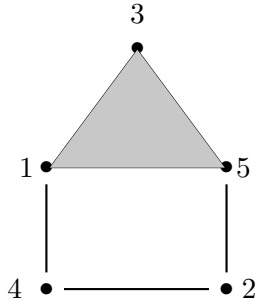
\begin{figure}[H]
    \centering
\begin{tikzpicture}[scale=0.46]
\coordinate (A1) at (0,0);
\coordinate (A2) at (2.5,3.4);
\coordinate (A3) at (5,0);
\node (A10) at (-0.1,0) {$\bullet$};
\node (A20) at (2.5,3.44) {$\bullet$};
\node (A30) at (5.05,0) {$\bullet$};
\node (100) at  (-0.25, 0) {$1\,\,\, \, \,\,$};
\node (102) at (2.5,4.43) {$3$};
\node (101) at (5.18,0) {$\,\,\,\, \, 5$};
\definecolor{c1}{RGB}{200,200,200}
\definecolor{c2}{RGB}{200,200,200}
\definecolor{c3}{RGB}{200,200,200}

\draw (A1) -- (A2) -- (A3) -- (A1) -- cycle;

\shade [left color=c1,right color=c2] (A1) -- (A2) -- (A3) -- (A1) -- cycle;
\node (B0) at (-0.1,-3.5) {$\bullet$};
\node (B1) at (-1,-3.5) {$4$};
\node (B2) at (5.05,-3.5) {$\bullet$};
\node (B3) at (5.7,-3.5) {$2$};

\draw (A10) edge[ color=black!120, thick, ] (B0);
\draw (B0) edge[ color=black!120, thick, ] (B2);
\draw (B2) edge[ color=black!120, thick, ] (A30);
\end{tikzpicture}
   \caption{The geometric realization  $|\Delta_{P_5}|$.}
    \label{P5}
\end{figure}

\end{example}

\section{Burning homology }\label{S5}
\setcounter{equation}{0}

In this Section, we define \emph{homology groups   of $G$-burnings}  as simplicial homology groups of $\Delta_G$.  We describe properties of introduced homology groups and  give   several examples

Let $R$ be a unitary commutative ring and $\Delta=(V,F)$ be a simplicial complex.  For $n\geq  0$,   \emph{homology groups} $H_n(\Delta, R) \,$  of the complex $\Delta$ with coefficients in $R$ are defined using a functorially defined chain complex 
$\Lambda_*(\Delta)=\{\Lambda_n, \partial\}$  of $R$ modules (see, for example,  \cite{Hatcher}, \cite{Matousek}, \cite{Munkres}, \cite{Prasolov}). Recall that the module $\Lambda_n(\Delta)$ is  generated by $n$-dimensional simplexes of $\Delta$. The \emph{augmented} chain complex $\widetilde\Lambda_*(\Delta)$  
is the complex $\Lambda_*(\Delta)$ equipped with the \emph{augmentation homomorphism} 
$$
\varepsilon \colon \Lambda_0(\Delta)\to R \ \  \text{where} \  \ \varepsilon\left(\sum_ i n_i \sigma_i\right) =\sum_i n_i 
$$
where  $\sigma_i$ are  zero dimensional simplexes and   $n_i\in R$. The  homology groups of the  augmented chain complex are called \emph{reduced homology  groups} and are denoted by $\widetilde H_*(\Delta,  R)$.

\begin{definition}\label{d5.1} \rm  The \emph{burning homology}  of a graph $G$   is the simplicial  homology  of the simplicial complex $\Delta_G$. The corresponding homology groups are denoted  by $H_k(\mathcal B(G), R)$ where $k\geq 0$. 
\end{definition}

\begin{proposition}\label{p5.2} Let a graph $G$ be a burning extension of a graph $H$. Then the natural inclusion  $i\colon H\to G$ induces  homomorphisms of
homology groups  
\begin{equation}\label{5.1}
i_*\colon H_k(\mathcal B(H), R)\to H_k(\mathcal B(G), R)  \ \text{for} \ \ k\geq 0.
\end{equation} 
\end{proposition}
\begin{proof}  Follows from  the definition of homology groups and Proposition~\ref{p4.15}. 
\end{proof}

\begin{corollary}\label{c5.3}   Every inclusion of trees $i\colon H\to  G$  induces   
homomorphisms of
homology groups  
\begin{equation}\label{5.2}
i_*\colon H_k(\mathcal B(H), R)\to H_k(\mathcal B(G), R) \ \ \text{for} \ \ k\geq 0.
\end{equation} 
\end{corollary} 
\begin{proof} Follows from  the definition of homology groups and Corollary \ref{c4.17}. 
\end{proof} 

Now we present several results about burning homology groups for various classes of graphs.  These results  follows from the description of the corresponding configuration spaces obtained in Section \ref{S4} using the standard results of algebraic topology \cite{Dold, Hatcher, Hilton, MacLane,  May, Munkres, Prasolov}. 

\begin{proposition}\label{p5.4} For every graph $G$ and the one vertex graph $K_1$, 
$$
H_k(\mathcal B(G+K_1), \mathbb Z)=
\begin{cases} 
\mathbb Z& \text{for}  \ \  k=0, \\
                                                             0& \text{for} \ \ n\ne 0. \\
\end{cases}
$$ 
\end{proposition}
\begin{proof} By Theorem \ref{t4.20}(i), $\Delta_{G+K_1}=C(\Delta_G)$ where 
$C(\Delta_G)$  is the cone over $G$.  From now the result follows. 
\end{proof}

\begin{proposition}\label{p5.5} For any connected graph $G$ and the path  graph $P_2$, 
$$
\widetilde H_k(\mathcal B(G+P_2), \mathbb Z)=\widetilde H_{k-1}(\mathcal B(G), \mathbb Z)\ \  \text{for} \ \ k\geq 0.  
$$ 
\end{proposition}
\begin{proof}  By Theorem \ref{t4.20}(ii), $\Delta_{G+P_2}=S(\Delta_G)$ where 
$S(\Delta_G)$  is the suspension over $G$.  From now the result follows. 
\end{proof}

\begin{proposition}\label{p5.6}  Let $\Pi_n=nP_2\, (n\geq 1)$ be an iterated sum of the path graph $P_2$. Then 
$$
H_k(\mathcal B(\Pi_n), \mathbb Z)=\begin{cases} \mathbb Z & \text{for } \ \  k=0, n-1,\\
0, & \text{otherwise}.\\
\end{cases}
$$
\end{proposition}
\begin{proof} Follows from Theorem \ref{t4.21}. 
\end{proof}

\begin{example}\label{e5.7} \rm (i) Let $K_n\, (n\geq 1)$ be a complete graph with $n$ vertices. Then 
$$
H_k(\mathcal B(K_n), \mathbb Z)=\begin{cases} \mathbb Z^n& \text{for } \ \  k=0, \\
0, & \text{otherwise}.\\
\end{cases}
$$

(ii) Let   $\overline K_n$ be an edgeless graph with $n$ vertices.  
Then 
$$
H_k(\mathcal B(\overline K_n), \mathbb Z)=\begin{cases} \mathbb Z& \text{for } \ \  k=0, \\
0, & \text{otherwise}.\\
\end{cases}
$$

(iii) Let $P_n \, (n\geq 1)$ be a path graph with $n$ vertices. Then the nontrivial  burning homology groups with the coefficient ring $\mathbb Z$ for $1\leq n\leq 6$ are the  following:
$$
H_0(\mathcal B(P_1))=\mathbb Z, \ \ H_0(\mathcal B(P_2))=\mathbb Z^2, \ \ 
H_0(\mathcal B(P_3))=\mathbb Z^2, 
$$
$$ H_0(\mathcal B(P_4))=\mathbb Z, \ \ 
H_0(\mathcal B(P_5))=\mathbb Z, \ \ H_1(\mathcal B(P_5))=\mathbb Z,  \ \ 
H_0(\mathcal B(P_6))=\mathbb Z. 
$$
\end{example}

\bibliographystyle{abbrv}

Yuri Muranov: \emph{Faculty of Mathematics and Computer Science, University
of Warmia and Mazury in Olsztyn, Sloneczna 54 Street, 10-710 Olsztyn, Poland. }

e-mail: muranov@matman.uwm.edu.pl

\smallskip

Anna Muranova: \emph{Faculty of Mathematics and Computer Science, University
of Warmia and Mazury in Olsztyn, Sloneczna 54 Street, 10-710 Olsztyn, Poland. }

e-mail: anna.muranova@matman.uwm.edu.pl

\end{document}